\documentclass[a4paper, 12pt]{amsart}
\usepackage{amsmath}
\usepackage{a4wide}
\usepackage{amssymb}
\usepackage{amscd}

\newtheorem{thm}{Theorem}[section]
\newtheorem{cor}[thm]{Corollary}
\newtheorem{lem}[thm]{Lemma}
\newtheorem{prop}[thm]{Proposition}
\theoremstyle{definition}
\newtheorem{defn}{Definition}
\newtheorem{pro}{Question}
\newtheorem{ex}[thm]{Example}
\theoremstyle{remark}
\newtheorem{rem}[thm]{Remark}
\DeclareMathOperator{\N}{\mathbb{N}}
\DeclareMathOperator{\Z}{\mathbb{Z}}
\DeclareMathOperator{\F}{\mathbb{F}}

\begin{document}

\author[P. Danchev]{Peter Danchev}
\address{Department of Mathematics, Plovdiv University, Bulgaria}
\email{pvdanchev@yahoo.com}
\author[J. Matczuk]{Jerzy Matczuk}
\address{Institute of Mathematics, Warsaw University, Poland}
\email{jmatczuk@mimuw.edu.pl}

\title[$n$-Torsion Clean Rings] {$n$-Torsion Clean Rings  }

\maketitle

\begin{abstract} Let $n$ be an arbitrary natural number. The class of (strongly) $n$-torsion clean rings is introduced and investigated. Abelian $n$-torsion clean rings are somewhat characterized and  a complete characterization of strongly $n$-torsion clean rings is given in the case when $n$ is odd. Some open questions are posed at the end.
\end{abstract}

\section*{Introduction}

Everywhere in the text, all rings are assumed to be associative with unity. Our notations and notions are in agreement with those from \cite{L}. For instance, for such a ring $R$, $U(R)$ denotes the group of units, ${\rm Id}(R)$ the set of idempotents and $J(R)$ the Jacobson radical of $R$, respectively. Besides, the finite field with $m$ elements will be denoted by $\mathbb{F}_m$, and $\mathbb{M}_k(R)$ will stand for the $k\times k$ matrix ring over $R$; $k\in \mathbb{N}$. For an element $u$ of a group $G$, $o(u)$ will denote the order of $u$. The symbol $LCM(n_1,\ldots , n_k)$ will be reserved for the least common multiple of $n_1,\ldots, n_k\in \N$.

We will say a nil ideal $I$ of $R$ is nil of index $k$ if, for any $r\in I$, we have $r^k=0$ and $k$ is the minimal natural number with this property. Likewise, we will say that $I$ is nil of bounded index if it is nil of index $k$, for some fixed $k$.

Let us recall that a ring $R$ is said to be {\it clean} if, for every $r\in R$, there are $u\in U(R)$ and $e\in {\rm Id}(R)$ with $r=e+u$. If, in addition, the commutativity condition $ue=eu$ is satisfied, the clean ring $R$ is called {\it strongly clean}. These rings were introduced  by Nicholson in \cite{N} and \cite{Ni}. Both clean rings and their either specializations or generalizations are intensively studied since then (see, for example, \cite{BDZ}, \cite{C}, \cite{D}, \cite{DM}, \cite{Di}, \cite{M} and references within).

A decomposition $r=e+u$ of an element $r$ in a ring $R$ will be called {\it $n$-torsion clean} decomposition of $r$ if $e\in {\rm Id}(R)$ and $u\in U(R)$ is $n$-torsion, i.e. $u^n=1$. We will say that such a decomposition of $r$ is strongly $n$-torsion clean, if additionally $e$ and $u$ commute.

The aim of this article is to investigate in detail the following proper subclasses of (strongly) clean rings:

\begin{defn}
 \label{tors} A ring $R$ is said to be ({\it strongly}) {\it $n$-torsion clean} if there is $n\in \N$  such that every element of $R$ has a (strongly) $n$-torsion clean decomposition and $n$ is the smallest possible natural number with the above property.
\end{defn}

It is easy to see that boolean rings are precisely the rings which are (strongly) $1$-torsion clean. Thus the introduced above classes of rings can be treated as a natural generalization of boolean rings.

Let us notice that in \cite{D} the class of ({\it strongly) invo-clean} rings was investigated. In our terminology, (strongly) invo-clean rings are precisely rings which are either (strongly) $1$-torsion clean or (strongly) $2$-torsion clean.

It is clear that every clean ring having the unit group of bounded exponent $s$ is $n$-torsion clean for some $n$ with $1\leq n\leq s$. We will see below that $n$ has to divide $s$, but does not have to be equal to $s$. Let us also observe  that a homomorphic image of an $n$-torsion clean ring is always $m$-torsion clean, for some $m\leq n$. Hoverer, it is not clear whether $n$ is a multiple of $m$. Notice that finite rings, being always clean, are $n$-torsion clean for suitable $n$ and it would be of interest to compute $n$ for some classes of finite  rings; for instance, for matrix rings over finite fields.

In the present paper we mainly concentrate on the case of strongly $n$-torsion clean rings. Our work is organized as follows: The first short section is of introductory character  and it contains some basic observations and examples. Strongly $n$-torsion clean rings are   investigated in Section 2. In particular, it is shown in Theorem~\ref{pi} that such rings have to satisfy a polynomial identity of degree $2n$ and that their Jacobson radical is nil of bounded index.  Theorem \ref{ab} offers a description of such rings which are abelian. It appears, which seems to be slightly surprising, that when $n$ is odd, strongly $n$-torsion clean rings have to be commutative. Their precise description is given in the subsequent Theorem~\ref{comm}. We finish off with some  open  questions.

\section{Preliminaries and Examples}

We begin with the following simple but useful observation. Its proof is provided for the sake of completeness.

\begin{lem}\label{descrip}
Let $R$ be a  (strongly) $n$-torsion clean ring. Then there exist a finite number of elements $r_1,\ldots, r_k\in R$ with (strongly) clean decompositions $r_i=e_i+u_i$, $1\leq i\leq k$, such that $n=LCM(o(u_1),\ldots,o(u_k))$. In particular:
\begin{enumerate}
   \item  When the group $U(R)$ has finite exponent $s$, then $n$ divides $s$.
   \item  When $R$ is commutative, then $U(R)$ contains an element of order $n$.
\end{enumerate}
\end{lem}

\begin{proof} For $r\in R$, let us set
$$
r_{min}=\min\{o(u)\mid r=e+u \mbox{ is a (strongly) $n$-torsion clean   decomposition of $r$ and $o(u)$ divides $n$}\}.
$$
Then each $r_{min}$ divides $n$. Thus    $LCM(r_{min}\mid r\in R)$ exists and also divides $n$. Moreover, we can pick   elements $r_1,\ldots, r_k\in R$ such that  $LCM(r_{min}\mid r\in R)=LCM( {r_1}_{min},\ldots,{r_k}_{min})$. The minimality of $n$ gives $LCM(r_{min}\mid r\in R)=n$. This completes the proof of the main statement.
Subsequently, (1) and (2) follow.
\end{proof}

It is well known that $1+J(R)\subseteq  U(R)$. In the class of rings for which the equality holds, the notation of $n$-torsion clean rings boils down to rings $R$ for which the unit group $U(R)$ is of finite exponent $n$. Indeed, we have:

\begin{prop}\label{U(R)} Let $R$ be a ring and $n\in \N$. Then:
\begin{enumerate}
\item  If $r\in J(R)$, then the unit $1+r$ has exactly one clean decomposition.
\item  Suppose $U(R)=1+J(R)$. Then the following two conditions are equivalent:
\begin{enumerate}
\item  $R$ is (strongly) $n$-torsion clean.
\item  $R$ is (strongly) clean and the group $U(R)$ is of finite exponent $n$.
\end{enumerate}
\end{enumerate}

Moreover, if one of the equivalent conditions holds, then $R/J(R)$ is a boolean ring.
\end{prop}

\begin{proof}
(1) Let $r\in J(R)$. Observe that if $1+r=e+u$ is a clean decomposition of $1+r$, then $1-e=u-r\in {\rm Id}(R)\cap U(R)=\{1\}$, that is, $e=0$. This implies that $1+r$ has the unique clean decomposition $r+1=0+(1+r)$.

(2) Suppose $R$ is (strongly) $n$-torsion clean and $u\in U(R)=1+J(R)$. Then, by (1), $u=0+u$ is the only clean decomposition of $u$ and $u^n=1$ follows, i.e. $U(R)$ is of finite exponent $s\leq n$.

Conversely suppose that $R$ is (strongly) clean and $U(R)$ is a group of finite exponent $s$. Then it is clear that $R$ is $n$-torsion clean ring, for some $n\leq s$. This yields the equivalence $(a)\Leftrightarrow (b)$.

Since units always lift modulo the Jacobson radical, we have $U(R/J(R))=\{1\}$. If $R$ is strongly $n$-torsion clean, then   $R/J(R)$ is $m$-torsion clean for some $m\leq n$. The above yields that $m=1$, i.e. $R/J(R)$ is a boolean ring.
\end{proof}

Notice that the ring  $T_m(\F_2)$ of all upper triangular $m\times m$ matrices over the field $\F_2$ is clean, its Jacobson  radical $J$ consists of all strictly upper triangular matrices and $U(T_m(\F_2))=1+J$. Thus, with Proposition~\ref{U(R)} at hand, we deduce:

\begin{ex}\label{tmr}
Let $m\in \N$ and let $k$ be the smallest nonnegative integer such that $m\leq 2^k$. Then the ring $T_m(\F_2)$ is (strongly) $2^k$-torsion clean.
\end{ex}

Recall that a ring $R$ is uniquely clean if every element of $R$ has a unique clean presentation. Such rings were characterized in \cite{NZ} as those abelian rings $R$ such that $R/J(R)$ is boolean (whence $U(R)=1+J(R)$) and idempotents lift modulo $J(R)$. Notice that, as idempotents always lift modulo nil ideals, every ring $R$ such that  $R/J(R)$ is boolean and $J(R)$ is nil must be clean. Therefore, the above proposition also gives the following corollary. Its second statement  generalizes  Example \ref{tmr}.

\begin{cor}\label{uniquely clean}
\begin{enumerate}  \item  Let $R$ be a uniquely clean ring. Then $R$ is $n$-torsion clean if and only if $U(R)$ is of exponent $n$;
\item  Let $R$ be a ring such that $R/J(R)$ is boolean and $J(R)$ is nil of bounded index. Then $R$ is    $n$-torsion clean, where $n$ is the exponent of $U(R)$. Moreover, $n$ is a power of 2.
\end{enumerate}
\end{cor}

\begin{proof}
(1) being an immediate consequence of the preceding discussion, let the ring $R$ be as in $(2)$. Then $R$ is a $UU$ ring (i.e. all units are unipotent) and so \cite[Theorem 3.4 (2)]{DL} implies that $U(R)$ is a $2$-group. Now the thesis is a simple consequence of Proposition~\ref{U(R)}.
\end{proof}

 Proposition~\ref{U(R)} demonstrates that, from the point of view of $n$-torsion clean property,   rings with $U(R)=1+J(R)$ are, in some sense, not too interesting. The situation when the ring is Jacobson semisimple  and has non-trivial group of units is much more interesting. The next example is of such nature and it shows that a ring can be $n$-torsion clean with $n$ strictly smaller than the exponent of the group $U(R)$.

\begin{ex}\label{nc} Let $R=\mathbb{M}_2(\mathbb{F}_2)$. Then $R$ is $2$-torsion clean and strongly $6$-torsion clean.
\end{ex}

\begin{proof} The ring $R$ is nil clean by virtue of \cite{BCDM}. Since the index of nilpotence of elements of $R$ is at most $2$,  \cite[Corollary 2.11]{D} implies that $R$ invo-clean which is not boolean, so that it is $2$-torsion clean. The above can be also checked by   direct computations. For instance, the unit $r=\left(
                    \begin{array}{cc}
                      1 & 1 \\
                      1 & 0 \\
                    \end{array}
                  \right)
$ of order $3$ has a $2$-torsion clean decomposition $r=\left(
                                            \begin{array}{cc}
                                              1 & 0 \\
                                              0 & 0 \\
                                            \end{array}
                                          \right) + \left(
                                                      \begin{array}{cc}
                                                        0 & 1 \\
                                                        1 & 0 \\
                                                      \end{array}
                                                    \right)
$ (but it does not have strongly $2$-torsion clean decomposition). It is also easy to make elementary computations showing that $R$ is strongly $6$-torsion clean, which we leave to the reader.
\end{proof}

Let us notice that the unit group of $\mathbb{M}_2(\F_2)$ is isomorphic to the symmetric group $\mathcal{S}_3$.

\begin{prop}\label{matrices}  Let $m, k\in \N$ be such that $m\leq 2^k$. Then $\mathbb{M}_m(\mathbb{F}_2)$ is $n$-torsion clean for some natural $n\leq 2^k$.
\end{prop}

\begin{proof} Set $R=\mathbb{M}_m(\mathbb{F}_2)$. It is known (cf. \cite[Theorem 3]{BCDM}) that $R$ is a nil-clean ring. Thus every element $x$ of $R$ can be written as $x=e+z$ with $e=e^2$ and $z^m=0$, as the index of nilpotence of elements in $R$ is bounded by $m$. Now we can write $x=(1+e)+(1+z)$, where $1+e$ is an idempotent and $(1+z)^{2^k}=1+z^{2^k}=1$, as $m\leq 2^k$. This enables us to conclude that every element of $R$ has $2^k$-torsion clean decomposition. In particular, $R$ is $n$-torsion clean for some $n\leq 2^k$.
\end{proof}

With the help of this proposition, we derive:

\begin{ex}\label{ex mat} The rings $\mathbb{M}_3(\F_2)$ and $\mathbb{M}_4(\F_2)$, in view of Proposition~\ref{matrices}, are $n$-torsion clean for some $n\leq 2^2$. The rings are, however, not invo-clean by virtue of \cite[Corollary 2.11]{D}, and so $n\in \{3,4\}$.
\end{ex}

The linear group $GL(3,\F_2)$ is the unit group of $\mathbb{M}_3(\F_2)$. The group is known to be simple of order $168$ and exponent $84$.

\section{Strongly $n$-torsion clean rings}

The following technical lemma is crucial for our further considerations.

\begin{lem}\label{eq} Suppose that $R$ is a ring and the element $a\in R$ possesses strongly $n$-torsion clean decomposition. Then the equality $(a^n-1)((a-1)^n-1)=0$ holds.
\end{lem}

\begin{proof} Let $a=e+v$ be a strongly $n$-torsion clean decomposition of $a$. Since $e,v$ commute and $v^n=1$, we deduce: $$a^n-1=(e+v)^n-1=\sum _{i=0}^n \binom{n}{i}e^iv^{n-i}-1=\sum _{i=1}^n \binom{n}{i}e^iv^{n-i}\in Re.$$  This implies that
\begin{equation}\label{a}
(a^n-1)(1-e)=0\; \mbox{ and, consequently, }\;a^n-1=(a^n-1)e.
\end{equation}
Using   $ve=(a-1)e$, we also have $1=v^n=(a-e)^n=(a-1)^ne-a^ne+a^n$. This yields
\begin{equation}\label{b}
a^n-1=(a^n-(a-1)^n)e.
\end{equation}

Applying (\ref{b}) and the second equation of (\ref{a}), we get $((a-1)^n-1)e=0$. This equality and the first equation of (\ref{a}) now give together that $(a^n-1)((a-1)^n-1)=(a^n-1)((a-1)^n-1)(e+(1-e))=0$, as desired.
\end{proof}

The following assertion is pivotal.

\begin{lem}\label{reduced} Let $n\in  \mathbb{N}$ and let $R$ be a ring satisfying the identity $(x^n-1)((x-1)^n-1)=0$. Then:
 \begin{enumerate}
   \item  $\mbox{\rm char} (R):=|1\cdot \Z|$ is finite and $J(R)$ is a nil ideal;
   \item  If $n$ is odd, then $R$ is a reduced ring of characteristic $2$ and $J(R)=0$;
   \item If $R$ is an algebra over a field $F$, then either $R$ is   abelian   (i.e. all idempotents of $R$ are central) or $\mbox{\rm char}(F)$ divides $n$.
 \end{enumerate}
\end{lem}

\begin{proof} Set $\phi(x)=(x^n-1)((x-1)^n-1)\in \Z[x]$.

(1) Substituting $x=3\cdot 1$ in the identity $\phi(x)=0$, we see that there exists $0\ne m\in \Z$ such that $1\cdot m=0$ in $R$. This shows that characteristic $\mbox{\rm char}(R)$ of $R$ is finite.

Let  $r\in J(R)$. Assume that $r$ is not nilpotent. Then the multiplicatively closed set $S=\{r^k\mid k\in \N\}$ does not contain $0$. Let $P$ be a maximal ideal of $R$ in the class of all ideals having empty intersection with $S$. So, $P$ is a prime ideal of $R$, the ring $\bar R=R/P$ satisfies the same identity as $R$ does and $\bar r=r+P\in J(\bar R)$. Moreover $\bar r$ is not nilpotent, as $S\cap P=\emptyset$. Thus, eventually replacing $R$ by $\bar R$, we may additionally assume that the ring $R$ is  prime. Then, its subring $1\cdot \Z=F$ of $R$ is a domain. By the first part of the proof, $1\cdot \Z$ is finite, so it is a field. This means that the element $r\in J(R)$ is algebraic over the field $F$ and, as such, has to be nilpotent (cf. \cite[Proposition 4.18]{L}). This contradicts the choice of $r$ and shows that every element of $J(R)$ is nilpotent.

(2)  Suppose $n$ is odd.  Then, substituting $x=0$ in the identity $\phi(x)=0$, we obtain $2=0$, i.e. $\mbox{\rm char}(R)=2$.

Let $r\in R$ be such that $r^2=0$. If $n=1$, then the identity $\phi(x)=0$ shows that $(r-1)r=0$ and $r=0$ follows immediately, as $r-1$ is invertible.

Suppose now that $n\geq 3$. Notice that  $(r-1)^2=1$. Thus, as  $n$ is odd, $(r-1)^n=(r-1)$. Therefore, $r=\phi(r)=0$. This shows that $R$ has no nonzero nilpotent elements, i.e. $R$ is reduced. Then also $J(R)=0$ as, by (1), $J(R)$ is a nil ideal.

(3) Suppose $R$ is an algebra over a field $F$. By (1), $\mbox{\rm char}(F)=p\ne 0$.  If  $n$ is odd then, using (2), $R$ is a reduced ring, so it is abelian. Suppose now that $n$ is even and $p$ does not divide $n$.  Thus   $1\cdot n$  is invertible in $R$. Let $e=e^2, r\in R$. Substituting $x:=er(1-e)$ in the identity $\phi(x)=0$ and using the fact that $n$ is even, we obtain $0=((er(1-e))^n-1)((er(1-e)-1)^n-1)=ner(1-e) $ and thus the equality $er(1-e)=0$ follows. Similarly $(1-e)re=0$. The above shows that every idempotent $e$ of $R$ is central, provided that $\mbox{\rm char}(R) $ does not divide $n$. This completes the proof of the lemma.
\end{proof}

\begin{rem}\label{rem nt} In regard to point (2) stated above, a routine argument demonstrates that when $R$ is an $n$-torsion clean ring and $n$ is odd, then $J(R)=0$ and $\mbox{\rm char}(R)=2$. Indeed, let $0=f+v$ be an $n$-torsion clean decomposition of $0$. Then $-f=(-f)^n=v^n=1$ and  $\mbox{\rm char}(R)=2$ follows.  Now, Proposition~\ref{U(R)}(1) yields that $1=(r-1)^n=\sum_{i=0}^n\binom{n}{i}(-1)^{n-i}r^i $, for any $r\in J(R)$. As $n$ is odd, this equation gives $0=rw$, where $w= \sum_{i=1}^n\binom{n}{i}(-1)^{n-i}r^{i-1}\in 1 +J(R)$ is invertible in $R$, i.e. $r=0$, as required.
\end{rem}

Now we are ready to establish the following theorem.

\begin{thm}\label{pi} Let $n\in \N$. Suppose $R$ is a strongly $n$-torsion clean ring. Then:
 \begin{enumerate}
   \item  $R$ is a PI-ring satisfying the polynomial identity $(x^n-1)((x-1)^n-1)=0  $;
   \item $R$ has finite characteristic   $\mbox{\rm char} (R)=|1\cdot \Z|$;
   \item  $J(R)$ is a nil ideal of   index   smaller than   $(\mbox{\rm char} (R))^n$;
   \item When $n$ is odd, then $R$ is a reduced ring of characteristic 2 and $J(R)=0$;
   \item If $R$ is an algebra over a field $F$, then:\\
   (i) $J(R)$ is a nil ideal of index bounded by $n$;\\
   (ii) either $R$ is abelian  (i.e. all idempotents of $R$ are central) or $\mbox{\rm char} (F)$ divides $n$.
 \end{enumerate}

\end{thm}

\begin{proof} The first statement is a direct consequence of Lemma~\ref{eq}. Notice that, in virtue of Lemma~\ref{reduced}, for completing the proof  it remains  only to show that $J(R)$ is nil of index bounded as indicated in the theorem.

  Let $r\in J(R)$. We claim that $r^{( {\rm char}(R))^n}=0$. By Lemma~\ref{reduced}(1), $r$ is nilpotent. Furthermore, Proposition~\ref{U(R)} (1) shows that the unit $1+r$ has exactly one clean presentation. Thus $(1+r)^n=1$ follows, as $R$ is $n$-torsion clean. Therefore $r^n\in S=(1\cdot \Z)[r]=(1\cdot \Z)r^{n-1} +\ldots+(1\cdot \Z)$. By (1), $ 1\cdot \Z$ is a finite ring with $c:=\mbox{\rm char}(R)$ elements. Hence the ring $S$ is finite and has at most $c^n $ elements. As $r\in S$ is nilpotent, its index of nilpotence has to be smaller than $|S|\leq c^n$ (to argue this, just consider the set $A\subseteq S$ of all powers of the element $r$ and show that $|A|$ is the nilpotence index of $r$). This gives (3).

  Suppose now that $R$ is an algebra over a field $F$. Then $S$ defined as above is, in this case, a finite dimensional algebra over $\F_p=1\cdot \Z\subseteq F$ of dimension not bigger than $n$. The dimension argument applied to the sequence of subspaces $S\supseteq Sr\supseteq Sr^2\supseteq \ldots$ shows that $r^n=0$, when $r\in S $ is  nilpotent. This yields  (5)(i) and completes the proof of the theorem.
\end{proof}

It is an important open question (see \cite[Question 2]{Ni}) whether strongly clean rings are Dedekind finite. Since PI rings are Dedekind finite, the above theorem gives immediately the following corollary:
\begin{cor}
 Strongly $n$-torsion clean rings are Dedekind finite.
\end{cor}
\begin{cor}
 Let $R$ be a strongly $n$-torsion clean ring. If $R$ is a finitely generated algebra over a central noetherian subring, then $J(R)$ is nilpotent.
\end{cor}
\begin{proof}
 By Theorem \ref{pi}(1), $R$ satisfies a monic polynomial identity and $J(R)$ is a nil ideal. Now the thesis is a direct consequence of \cite[Theorem 2.5]{B}.
\end{proof}
The following example shows that, in general, Jacobson radical of strongly $n$-torsion clean rings does not have to be nilpotent.
\begin{ex}\label{jacobson exm}
 Let $F=\F_{p^k}$ and $F[X]$  be   the polynomial ring in  infinitely many commuting indeterminates from the set $X$.   Set $R=F[X]/I$, where $I$ is the ideal of $F[X]$ generated by all elements  $x^p$, $x\in X$. Then $R$ is a local ring and its Jacobson radical is not nilpotent.  Making use of Propositions \ref{U(R)}(1) and \ref{fields}(1), one can  easily    check that $R$ is $p(p^k-1)$-torsion clean.
\end{ex}

In Theorem~\ref{comm} stated in the sequel we will present a complete characterization of strongly $n$-torsion clean rings in the case when $n$ is odd. For doing so, the following proposition, which gives a characterization of strongly $n$-torsion clean rings which are subdirect products of fields, is needed.

\begin{prop}\label{fields}
\begin{enumerate}
\item  Let $F$ be a field. Then $F$ is $n$-torsion  clean if and only if $F$ is   finite and $n=|F|-1$.
\item  A product of fields $\mathbb{F}_{p_1^{k_1}}\times \ldots \times\mathbb{F}_{p_t^{k_t}}$ is $n$-torsion clean, where $n= LCM(p_1^{k_1}-1,\ldots , p_t^{k_t}-1)$;
\item  A product $ \prod_{i\in I}F_i$ of fields is $n$-torsion clean if and only if all fields $F_i$, $i\in I$, are finite, $LCM(|F_i|-1\mid i\in I)$ exists and is equal to $n$;
\item Let $R$ be a subdirect product of fields $F_i$, $i\in I$. Then $R$ is  $n$-torsion clean if and only if $\prod_{i\in I}F_i$ is $n$-torsion clean.
\end{enumerate}
\end{prop}

\begin{proof}(1) Notice that any finite field $F$ is $n$-torsion clean for some divisor $n$ of $|F|-1$. On the other hand, if $F$ is any field which is $n$-torsion clean then, by Theorem \ref{pi}, every element of $F$ is a root of the polynomial $(x^n-1)((x-1)^n-1)\in F[x]$, so $|F|$ is finite and $|F|\leq 2n$. Suppose that $F$ is a finite $n$-torsion clean field and let $|F|-1=l\cdot n$. By what we have just shown it follows that $l=\frac{|F|-1}{n}\leq 2-\frac{1}{n}<2$ and so $l=1$ holds, i.e. $s=|F|-1$, as required.

(2) Let $T=\mathbb{F}_{p_1^{k_1}}\times \ldots \times\mathbb{F}_{p_t^{k_t}}$ and let $n$ be as defined in (2). Notice that $n=\max\{o(u)\mid u\in U(T)\}$ and the order of any $u\in U(T)$ divides $n$. Therefore, $T$ is $m$-torsion clean for some $m\leq n$.

For showing that $n=m$, it suffices to show that $n_i=p_i^{k_i}-1$ divides $n$, for any $1\leq i\leq t$. Note that, by (1), $F_i=\F_{p_i^{k_i}}$ is $n_i$-torsion clean. Furthermore, using Lemma~\ref{descrip}, we can pick elements $r_1, \ldots ,r_s\in T$ and their clean decompositions $r_j=e_j+u_j$, $1\leq j\leq s$, such that $m=LCM(u_1,\ldots, u_s)$. For a fixed $1\leq i\leq t$ consider the set $\{\pi_i(r_1), \ldots, \pi_i(r_s)\}\subseteq F_i$, where $\pi_i$ denotes the canonical projection of $R$ onto $F_i$. Then, for every $a\in F_i$, $a$ can be presented as $e+u$ with $u^{z_i}=1$, where $z_i=LCM(o(\pi_i(u_1)), \ldots, o(\pi_i(u_s)))$. Thus $n_i\leq z_i$. As $z_i$ is a $LCM$ of orders of elements in a cyclic group $U(F_i)$ of order $n_i$, we also deduce that $z_i\leq n_i$, i.e. $z_i=n_i$. This implies that $n_i$ divides $m$, for any $1\leq i\leq t$, as desired.

(3) Suppose the product $ \prod_{i\in I}F_i$ is $n$-torsion clean. Then every field $F_i$ is a homomorphic image of
$\prod_{i\in I}F_i$. Thus, owing to (1), each $F_i$ is a finite field. If $LCM(|F_i|-1\mid i\in I)$ would not exist, then there would exist indexes $i_1,\ldots, i_k\in I$ such that $m=LCM(|F_{i_1}|-1, \ldots, |F_{i_k}|-1 )>n$. However, in virtue of (2), $T=F_{i_1}\times \dots\times  F_{i_k} $ is $m$-torsion clean and $m\leq n$, as $T$ is a homomorphic image of $R$. Thus $LCM(|F_i|-1\mid i\in I)$ do exist and we can assume that $LCM(|F_i|-1\mid i\in I)=LCM(|F_{i_1}|-1, \ldots, |F_{i_k}|-1 )=m$. Then it is clear that $n\leq m$. Notice also that $m\leq n$, as $T$ is a homomorphic image of $R$, i.e. $n=m$. This gives (3).

(4) Let $R$ be a subdirect product of fields $F_i$, $i\in I$. Suppose $R$ is $m$-torsion clean. For every $i\in I$,   $F_i$ is a homomorphic image of $R$ so, with the aid of (1), the field $F_i$ is $n_i$-torsion clean, where $n_i=|F_i|-1$.  We also have $n_i\leq m$. Therefore, $ LCM(|F_i|-1\mid i\in I)$ exists and the statement (3) shows that $ \prod_{i\in I}F_i$ is $n$-torsion clean, where $n=LCM(|F_i|-1\mid i\in I)$. In particular, the order of any unit of $R$ divides $n$ and thus $m\leq n$ follows.

Let us fix $i\in I$ and let $F=F_i$ with $s=n_i$. Then, any $a\in F$ can be presented as $a=e+u$ with $u^m=1$. Let $k_1,\ldots k_s$ be orders of units in such presentations of all elements of $F$. Then, by construction,  $k=LCM(k_1,\ldots,k_s)$ divides $m$ and also divides $s=|F|-1$ (as $s$ is equal to the order of the group $U(F)$). In particular $k\leq s$. On the other hand, appealing to (1), $F$ is $s$-torsion clean and this forces that $s \leq k$. However, this shows that $k=s=n_i$ divides $m$. That is why, this means that, for any $i\in I$, $n_i=s$ divides $n$.    Consequently, $n=LCN(n_i\mid i\in I)$ divides $m$. By the first part of the proof $m\leq n$, so $n=m$ really follows.

Suppose now that $\prod_{i\in I}F_i$ is $n$-torsion clean and $R$ is a subdirect products of fields $F_i$, $i\in I$. To complete the proof, it is enough to show $R$ is $m$-torsion clean for some $m$. The statement (3) implies that the group $U(R)$ is of finite exponent, say $k$ is the exponent. Then, for any $r\in R$, $e=r^k $ is an idempotent, and $r=(1-e)+((e-1)+r)$ is a clean decomposition of $a$ with $((e-1)+r)^k=1$. This allows us to conclude that $R$ is $m$-torsion clean, for some $m\leq k$, as required.
\end{proof}

The following result, which is needed further in the text, is also of some independent interest. Before stating it, let us recall that idempotents lift modulo an ideal $J$ of $R$ if, for any $a\in R$ such that $a^2-a\in J$, there exists an idempotent $e\in R$ such that $e-a\in J$. If the idempotent $e$ is uniquely determined by the element $a$, then we say that idempotents lift uniquely modulo $I$.

It is known that idempotents lift modulo nil ideals, thus the following lemma applies when $J$ is a nil ideal of a ring $R$.

\begin{lem}\label{idemp}
Let $J\subseteq J(R)$ be an ideal of  $R$. Suppose that idempotents lift modulo $J$. Then the following    conditions are equivalent:
 \begin{enumerate}
   \item  $R$ is an abelian ring;
   \item  $R/J$ is an abelian ring and idempotents lift uniquely modulo $J$.
 \end{enumerate}
\end{lem}

\begin{proof} Let $\pi\colon R\rightarrow R/J$ denotes the canonical homomorphism.

$(1)\Rightarrow (2)$. Suppose the ring $R$ is abelian.  Since idempotents lift modulo $J$,   ${\rm Id}(R/J)=\pi({\rm Id}(R))$. Thus the ring $R/J$ is abelian, as $R$ is such. Let $e,f\in {\rm Id}(R)$ be such that $e-f\in J\subseteq J(R)$. Then, by \cite[Corollary 11]{KLM},    $e$ and $f$ are conjugate in $R$, i.e. there exists   $u\in U(R)$ such that $e=ufu^{-1}$. Hoverer, all idempotents of $R$ are central, so $e=f$. This, together with the assumption that idempotents lift modulo $J$ yield that idempotents lift uniquely modulo $J$.

$(2)\Rightarrow (1)$. The commutator of elements $a,b\in R$ will be denoted by $[a,b]:=ab-ba$. Suppose $(2)$ holds and let $e\in {\rm Id}(R)$, $r\in R$. Then $f=e+er(1-e)$ is also an idempotent and $[f,e]=er(1-e)$. By assumption  $R/J$ is abelian, so $\pi([f,e])=0$. This shows that  $er(1-e)\in J$.  Since  $\pi (e)=\pi (f)$ and, by assumption,  idempotents lift uniquely modulo $J$, we obtain $e=f$, i.e. $er(1-e)=0$. Now, replacing $e $ by $1-e$, we also have $(1-e)re=0$, for any $r\in R$. This  means that every idempotent $e$ of $R$ is central, i.e. $R$ is abelian, as required.
\end{proof}
We will need in the sequel the following direct application of  \cite[Theorem 3.2.]{DL}.
\begin{lem} \label{lem L} Let  $R$ be a ring and $u\in R$. Suppose that  $m:={\rm char}(R)$ is finite and $J(R)$ is a nil ideal  of  index $s+1$, where $s\geq 0$. If $ u^t-1\in J(R)$, then $u^{tm^s}=1$.
\end{lem}
\begin{proof} \cite[Theorem 3.2.]{DL} states that  if $R$ is a ring satisfying assumptions of the lemma, then  $(1-r)^{m^s}=1$, for any $r\in J(R)$.
  Now, if $u^t-1\in J(R)$, then there exists $r\in J(R)$ such that $u^t=1-r$ and  $u^{tm^s}=1$  follows.
\end{proof}
The above lemma give immediately the following corollary:
\begin{cor}\label{cor L}
Let $R$ be a ring of such that ${\rm char}(R)$ is finite and $J(R)$ is  nil of bounded index.  If the group $U(R/J(R))$ is of finite exponent, then so is $U(R)$. If additionally  $R/J(R)$ is clean (so $R$ is also clean, as units and idempotents lift modulo nil ideals), then   $R$ is  $n$-torsion clean, for some $n\in \N$.
  \end{cor}

\begin{cor}
Let $R$ be a ring of finite characteristic and $J$ a nil ideal of $R$ of bounded index. Then the following conditions are equivalent:
\begin{enumerate}
      \item  $R$ is an $n$-torsion clean ring, for some $n\in \N$.
     \item  $R/J$ is an $t$-torsion clean ring, for some $t\in \N$.
\end{enumerate}
\end{cor}
\begin{proof}
 Suppose $R/J$ is an $m$-torsion clean ring, for some $m\in \N$. Let $r\in R$. Since units and idempotents lift modulo $J$ we can find $e\in {\rm Id}(R)$ and $u\in U(R)$ such that $\bar r=\bar e+\bar u$ is an $t$-torsion clean decomposition of $\bar r$ in $R/J$, where $\bar r$ denotes the natural image of $r$ in $R/J$. By Lemma \ref{lem L}, $u^{}=1$, where $m={\rm char}(R)$ and $s+1$ is the nil index of the ideal  $J$. This implies that $R$ is $n$ torsion clean, for some $n\leq tm^s$.

 The reverse implication is clear.
\end{proof}
The following theorem offers a characterization of strongly $n$-torsion clean abelian rings (compare with Theorem \ref{pi}).

\begin{thm}\label{ab}
For a ring $R$, the following conditions are equivalent:
\begin{enumerate}
  \item There exists $n\in \mathbb{N}$ such that $R$ is an $n$-torsion clean abelian ring.
   \item  \begin{enumerate}
            \item  $\mbox{\rm char}(R)$ is finite;
            \item The Jacobson radical $J(R)$ is nil of bounded index;
            \item Idempotents lift uniquely modulo $J(R)$;
            \item  $R/J(R)$ is a subdirect product of  finite fields $F_i$,  where $i$ ranges over some index set  $I$,  such  that $LCM(|F_i|-1\mid i\in I)$ exists.
          \end{enumerate}
          \item $R$ is an abelian clean ring such that the unit group $U(R)$ is of finite exponent.
\end{enumerate}
\end{thm}

\begin{proof}
$(1)\Rightarrow (2)$. Suppose $R$ is an $n$-torsion clean abelian ring. Then, Theorem \ref{pi} guarantees that $\mbox{\rm char}(R) $ is finite and $J(R)$ is nil of finite index. In particular, $R$ has properties (a) and (b). Since $J(R)$ is a nil ideal, idempotents lift modulo $J(R)$ and, by Lemma \ref{idemp}, they lift uniquely, so (c) holds.

Finally, by Theorem~\ref{pi}(1), $R/J(R)$ satisfies the polynomial identity $\phi(x)=0$, where  $\phi(x)=(x^n-1)((x-1)^n-1)\in  \Z[x]$. Therefore, $R/J(R)$ is a subdirect product of primitive PI-rings, say $R/J(R)$ is a subdirect product of primitive rings $\{R_i\}_{i\in I}$, for some index set $I$. Let us fix $i\in I$. Then $R_i$, as a homomorphic image of $R$, also satisfies the identity $\phi(x)=0$. Consequently, by the classical Kaplansky's theorem (cf. \cite{R}), each $R_i$ has to be a  central simple algebra, finite dimensional over its center $C$. Notice that, as $\mbox{\rm char}(R)$ is finite, $C$ is a field of nonzero characteristic, say $\F_p\subseteq C$. Observe also that, by Lemma~\ref{idemp},   $R_i$ is an abelian ring. This implies that $R_i$ has to be   a division algebra over $\F_p$.    It is known (cf. \cite[Corollary from page 48]{F}) that every division algebra which is algebraic over a finite field is necessarily commutative. In particular, $R_i$ has to be a field. In fact, it is a finite field, as $R_i$ is contained in the spitting field of $\phi(x)\in\F_p[x]$. The above shows that $R/J(R)$ is a subdirect product of finite fields. Moreover, $R/J(R)$ is also, as a homomorphic image of $R$, strongly $n'$-torsion clean, for some $n'\leq n$.  Therefore, making use of Proposition~\ref{fields}, we see that $R$ satisfies the property (d). This completes the proof of the implication.

$(2)\Rightarrow (3)$. Suppose (2) holds. We know, by (d) and  Proposition~\ref{fields},  that  $R/J(R)$ is a clean ring  with the unit group $U(R)$ of finite exponent.
The property  (b) guarantee that $J(R)$ is a nil ideal and  Corollary \ref{cor L} yields that $R$ is a clean ring with the unit group $U(R)$ is of finite exponent. Finally,  properties (d),  (c) together with  Lemma \ref{idemp} imply that $R$ is an abelian ring.

The implication $(3)\Rightarrow (1)$ is obvious.
\end{proof}

In parallel to Theorem~\ref{ab}, one can state the following:

\begin{thm} \label{thm stn} For a ring $R$, the following conditions are equivalent:
\begin{enumerate}
  \item  $R$ is strongly $n$-torsion clean, for some $n\in \N$.
  \item  $R$ is strongly clean and $U(R)$ is of finite exponent.
\end{enumerate}
\end{thm}

\begin{proof} $(1)\Rightarrow (2)$. Suppose $R$ is strongly $n$-torsion.   Then clearly $R$ is strongly clean. Next, observe that Theorem~\ref{pi} implies that  $R$ is a PI-ring satisfying an identity of degree $2n$ and $J(R)$ is a nil ideal of bounded index. Using similar arguments as in the proof of Theorem~\ref{ab}, one can see that the quotient $R/J(R)$ is a subdirect product of a matrix rings, say $R_i=M_{m_i}(F_i)$, over finite fields $F_i$. Notice that, as $\mbox{\rm char}(R)$ is finite, the set of characteristics of fields from the set $\mathcal{F}=\{F_i \mid i\in I\}$ is finite and also the number of fields of a given characteristic $p$ is finite, as every such field is contained in the splitting field of a given polynomial of degree $2n$. Thus  there are only finitely many classes of isomorphic fields in the set  $\mathcal{F}$. Moreover, by the classical Amitsur-Levitzki's theorem (cf. \cite{R}), each $m_i$ is not grater than $n$, as every $R_i$ satisfies a polynomial identity of degree $2n$. Therefore, the unit group of the product $\prod_{i\in I}R_i$ is a group of finite exponent. By  Theorem~\ref{pi}, $ {\rm char}(R)$ is finite and $J(R)$ is a  nil ideal of bounded index.  Now,  we can apply   Lemma \ref{lem L} to obtain that the group $U(R)$  is of finite exponent.

 The implication $(2)\Rightarrow (1)$ is clear.
\end{proof}

We now have at our disposal all the necessary information to present a satisfactory structural characterization of strongly $n$-torsion clean rings, for all odd $n$.

\begin{thm}\label{comm}
Suppose  $n\in \mathbb{N}$ is odd. For a ring $R$, the following conditions are equivalent:
\begin{enumerate}
\item  $R$ is a strongly $n$-torsion clean ring;
\item There exist  integers $k_1,\dots ,k_t\geq 1$ such that $n=LCM(2^{k_1}-1, \ldots, 2^{k_t}-1)$ and $R$ is a subdirect product of copies of fields $\mathbb{F}_{2^{k_i}}$, $1\leq i\leq t$;
    \item  $R$ is a  clean ring in which orders of all units are odd, bounded by $n$ and there exists a unit of order $n$.
\end{enumerate}
\end{thm}

\begin{proof}
$(1)\Rightarrow (2)$. Suppose $R$ is a strongly $n$-torsion clean ring. Then, by Theorem \ref{pi} (4), $R$ is a reduced ring of characteristic 2 and $J(R)=0$. Thus, as every reduced ring is abelian, we can apply Theorem \ref{ab} to obtain that $R$ is a subdirect product of finite fields $F_i$ of characteristic 2, where $i\in I$, for some index set $I$. Now, Proposition~\ref{fields} completes the proof of the implication.

The reverse implication (2) $\Rightarrow$ (1) is a direct consequence of Proposition~\ref{fields}. The implication
$(2)\Rightarrow (3)$ is a tautology.

$(3)\Rightarrow (1)$. Let $R$ be as in (3). Then, as $(-1)^2=1$ and $R$ has no units of even order, $-1=1$, i.e., $\mbox{\rm char}(R)=2$. Let us observe that $R$ has to be reduced. Indeed, if $r^2=0$ for some $r\in R$, then $(1+r)^2=1+r^2=1$. Using again the fact that  $R$ has no units of even order, we get $r=0$. It is known that in a reduced ring all idempotents are central. Moreover, by assumption, $R$ is a clean ring and, as every unit of $R$ is of finite order bounded by $n$, the ring must be   strongly $m$-torsion clean, for some $m\leq n!$. Now, because orders of units are odd, $m$ has to be odd (as $u^{2k}=1$ yields $u^k=1$, when $o(u)$ is odd). Furthermore, bearing in mind the equivalence of statements (1) and (2), we conclude that $n=m$, as required.
\end{proof}

It is worth to mention certain slightly unexpected, non-trivial consequences of the above theorem. Namely, not every odd natural number $n$ can serve as torsion degree of strongly $n$-torsion clean rings and, for odd,    $n$-torsion clean rings  are always commutative.

Notice that, because every finite ring is clean, Theorem~\ref{comm} forces the following:

\begin{cor}\label{finite}
For a finite ring $R$ the following conditions are equivalent:
\begin{enumerate}
\item  $R$ is strongly $n$-torsion clean for some odd $n$;
\item $R$ has no units of even order;
\item $R$ is isomorphic to a finite direct product of fields of characteristic $2$.
\end{enumerate}
\end{cor}

\begin{proof}
By the  Chinese Remainder Theorem, any subdirect product of finite number of fields is isomorphic to a direct product of fields. Now, the corollary is a straightforward consequence Theorem~\ref{comm}.
\end{proof}

We close the paper with some problems of interest.

\begin{pro} The matrix ring $\mathbb{M}_n(\mathbb{F}_{2^k})$ is always $m$-torsion clean for some $m$. Compute $m$ in terms of $n$ and $k$; is $m=n$ if $k=1$?
\end{pro}

Recall that some basic observations related to the above problem can be found in Proposition~\ref{matrices} and Examples \ref{nc} and \ref{ex mat}. In particular  $M_2(\F_2)$ is 2-torsion clean and, when $n\in\{3,4\}$ then $M_n(\F_2)$ is $m$-torsion clean, where $2<m\leq 4$. Christian Lomp checked for us, with the help of SageMath,  that $n=m$ in the above cases.
%

We have seen in Theorem \ref{thm stn} that    strongly $n$-torsion clean rings have units group $U(R)$ of finite exponent. For odd $n$, by Theorem~\ref{comm}, $n=\mbox{exp}(U(R))$.
  Example~\ref{nc} shows also such equality in the case of the ring $M_2(\F_2)$. We were kindly informed by Pace Nielsen, that such equality also holds for $M_3(\F_2)$, i.e. $M_3(\F_2)$ is strongly 84-torsion clean. Notice also that   Example \ref{jacobson exm} offers yet another instance of equality $n= \mbox{exp}(U(R))$.

Thus we pose the following two questions.
\begin{pro}Let $R$ be a strongly $n$-torsion clean ring. Is it true that $n= \mbox{exp}(U(R))$?
\end{pro}

\begin{pro} Let $R$ be an $n$-torsion clean ring.  Is then necessary  $U(R)$ of finite exponent?
\end{pro}

For odd $n$, strongly $n$-torsion clean rings were characterized in Theorem~\ref{comm}. Besides, Theorem~\ref{ab} offers a description of strongly $n$-torsion clean rings with extra assumption that the  considered rings are abelian. So, we come to

\begin{pro}\label{st even}
Characterize strongly $n$-torsion clean rings, for even $n\in \N$.
\end{pro}
 If $R$ is not abelian, then   Theorem \ref{pi} (5) and arguments used in the proof of Theorem \ref{thm stn} show  that, modulo the Jacobson radical (which is nil of bounded index), Question \ref{st even} essentially reduces to the  investigation of matrix rings over finite fields of characteristic dividing $n$.

It is also worthwhile noticing that (strongly) $2$-torsion clean rings were classified in \cite{D} under the name (strongly) invo-clean rings by using another approach. In fact, $R$ is strongly invo-clean if  and only if  $R\cong R_1\times R_2$, where $R_1$ is a ring for which $R_1/J(R_1)$ is boolean with $z^2=2z$ for every $z\in J(R_1)$, and $R_2$ is a ring which can be  embedded in a direct product of copies of the field $\F_3$.

\vskip2pc
\noindent \bf Acknowledgments. \rm   We would like to thank   Christian Lomp and Pace Nielsen for helpful discussion and support in computations.

\end{document}